\documentclass[12pt]{article}
\usepackage{amsmath,amssymb, amsthm, multicol, mathtools}
\usepackage[hidelinks]{hyperref}
\usepackage{tikz-cd}
\usepackage{enumitem, comment}
\usetikzlibrary{patterns, arrows}
\newcommand{\ds}{\displaystyle}

\newcommand{\N}{\mathbb{N}}
\newcommand{\C}{\mathbb{C}}
\newcommand{\R}{\mathbb{R}}
\newcommand{\D}{\mathbb{D}}

\newcommand{\re}{\textup{Re}}
\newcommand{\Hol}{\textup{Hol}}
\newcommand{\Aut}{\textup{Aut}}
\newcommand{\Int}{\textup{int}}

\newtheorem{lemma}{Lemma}[section]
\newtheorem{theorem}[lemma]{Theorem}

\newtheorem{conjecture}[lemma]{Conjecture}
\newtheorem{proposition}[lemma]{Proposition}
\newtheorem{defnition}[lemma]{Definition}
\newtheorem{remrk}[lemma]{Remark}
\newtheorem{exmple}[lemma]{Example}
\newenvironment{definition}{\begin{defnition} \em}{\end{defnition}}
\newenvironment{remark}{\begin{remrk} \em}{\end{remrk}}
\newenvironment{example}{\begin{exmple} \em}{\end{exmple}}

\title{Bounded domains on Kobayashi hyperbolic manifolds covering compact complex manifolds}
\author{Nicholas Newsome}
\date{}

\begin{document}

\maketitle

\begin{abstract}

The lack of a uniformization theorem in several complex variables leads to a desire to classify all of the simply connected domains. We use established computational methods and a localization technique to generalize a recently-published classification. In particular, we show that if a domain with $C^{1,1}$ boundary on a Kobayashi hyperbolic complex manifold contains a totally real boundary point and covers a compact manifold, then its universal cover must be the Euclidean ball. 
\end{abstract}

\section{Introduction}
\label{IntroSec}
The Riemann mapping theorem states that every proper, simply connected open subset of $\C$ is biholomorphic to the disk. However, this is not the case in higher dimensions. This means that the study of holomorphic functions of several complex variables must depend on the domains themselves. In $\C$, the uniformization theorem allows us to reduce analytic questions about planar domains to analytic questions on the disk. The absence of a higher-dimensional analogue of the uniformization theorem gives rise to a unique challenge in several complex variables: Classifying the simply connected domains. This problem is the inspiration for the current work.

The main result of this paper is the following:
\begin{theorem}
\label{dissthm}
Let $M^n$ be a taut complex manifold, and let $\Omega\subset M$ be a subdomain with nonempty boundary, and assume $\partial\Omega$ is $C^{1,1}$. Suppose there exists a totally real boundary point $p\in\partial\Omega$. Suppose further that $\Omega$ covers a compact complex manifold. Then $\Omega$ is biholomorphic to the Euclidean ball in $\C^n$.
\end{theorem}

Theorem \ref{dissthm} weakens slightly a condition of a result by Cheung et al. \cite{CFKW}, and uses a result of Zimmer \cite{Zimmer}, listed below. Theorem \ref{dissthm} also extends a result of Wong \cite{Wong77}.

\begin{proposition}[Proposition 3.1 in \cite{CFKW}]
\label{CFKWProp3.1}
Let $\Omega$ be a relatively compact subdomain of an $n$-dimensional hyperbolic complex manifold $M$. If $\Omega$ admits a compact quotient, and $\partial \Omega$ is smooth and strictly pseudoconvex near a point $p\in\partial \Omega$, then $\Omega$ is biholomorphic to the ball.
\end{proposition}

\begin{theorem}[Theorem 1.1 in \cite{Zimmer}]
\label{ZimThm1.1}
Suppose $\Omega\subset\C^n$ is a bounded domain which covers a compact manifold. If $\partial\Omega$ is $C^{1,1}$, then $\Omega$ is biholomorphic to the unit ball in $\C^n$.
\end{theorem}

\begin{theorem}[Wong \cite{Wong77}]
\label{WongC2}
Suppose $\Omega\subset\C^n$ is a bounded domain which covers a compact manifold. If $\partial\Omega$ is $C^2$, then $\Omega$ is biholomorphic to the unit ball.
\end{theorem}

Theorem \ref{WongC2} was proved by Wong for strongly pseudoconvex domains. Rosay \cite{Rosay} later extended the result to any bounded domain with $C^2$ boundary.

The proof of Proposition \ref{CFKWProp3.1} -- as well as other rigidity results similar to it and Theorem \ref{WongC2} -- relies on the idea that the interior complex geometry of the domain near a strongly pseudoconvex boundary point is close to that of the ball. Moreover, every bounded domain in $\C^n$ with $C^2$ boundary has at least one strongly pseudoconvex point on the boundary. Then, since the domain covers a compact complex manifold, the interior geometry must be \textit{everywhere} close to that of the ball; the theorem follows through use of a limiting argument.

Zimmer's contribution for the $C^{1,1}$ case is to notice that localizing around a strongly pseudoconvex point is no longer possible. Zimmer's proof of Theorem \ref{ZimThm1.1} uses a rescaling technique of Frankel \cite{Frankel89} to show that the domain $\Omega\subset\C^n$ is biholomorphic to a domain $D\subset\C^n$ containing a one-parameter subgroup. It follows by a theorem of Frankel and Nadel \cite{Frankel}, \cite{Nadel} that $\Omega$ is a bounded symmetric domain. Zimmer then uses the geometry of the rescaled domain to show that $\Omega$ is the unit ball. 

Our proof follows much of the same logic as Zimmer's. The main difference in the arguments is in the construction of the rescaling maps. For this construction, we use a localization technique of Lin and Wong \cite{LW} (see Section \ref{LWSec}) to force the domain $\Omega\subset M$ to be considered as a domain in $\C^n$. In their paper, Lin and Wong accomplish this by choosing holomorphic coordinates near a boundary point so that the local geometry is Euclidean. The boundary then assumes a normalized form. After an additional scaling, the localized domain converges to a domain in $\C^n$. Normal family arguments and other classical tools from several complex variables are now applicable to the original manifold. This construction is crucial in our proof of Theorem \ref{dissthm}, particularly in the proof of Proposition \ref{MRProof1}. In this result, we introduce new sequences of automorphisms in order to use the localization technique, and so our argument requires careful verification that Zimmer's computations remain valid. 

Since we consider domains of a general complex manifold as opposed to domains in $\C^n$, it is natural to expect some differences between the hypotheses of our main result and Zimmer's (Theorem \ref{ZimThm1.1}). In particular, the tautness assumption is necessary in order to use the localization technique. Some of these differences are slight, and so we note that our argument is, in some places, identical to Zimmer's, which we indicate in the presentation of our proof.

The remainder of this paper is laid out as follows. In Section \ref{GeoSec}, we discuss relevant geometric properties of domains in $\C^n$. In Section \ref{LWSec}, we describe the aforementioned localization technique \cite{LW} that will be instrumental in the construction of our rescaling maps. The proof of Theorem \ref{dissthm} is given in Section \ref{ProofSec}. Finally, in Section \ref{ConclusionSec} we present a potential consequence of our main result.

\section{Geometry of Domains}
\label{GeoSec}

\subsection{Hyperbolic Manifolds}
Here we introduce some terminology associated with hyperbolic manifolds as it pertains to our main result. A thorough discussion on hyperbolic manifolds can be found in \cite{KobayashiHyper}.

\begin{definition}
\label{kobayashimetric}
Let $M$ be a complex manifold. For $x\in M$ and $v\in T^{(1,0)}M$, the \textit{Kobayashi-Royden pseudometric} is defined by
$$k_{M}(x,v)=\inf\left\{\frac{1}{\lambda}\mid f\in\Hol(\mathbb{D},M),f(0)=x, f'(0)=\lambda v,\lambda>0\right\}.$$
\end{definition}

We let $d_{M}^K(x,y)$ denote the induced Kobayashi pseudodistance of $M$. It has the explicit form
$$d_{M}^K(z,w)=\inf_{\gamma}\int_0^1k_{M}(\gamma(t),\gamma'(t))dt.$$
Here, $z,w\in M$ and $\gamma:[0,1]\to M$ is a piecewise $C^1$ curve starting at $z$ and ending at $w$.

\begin{definition}
\label{hyperbolicmanifolddef}
A complex manifold $M$ is called \textit{(Kobayashi) hyperbolic} if $d_M^K$ is a distance. $M$ is called \textit{completely hyperbolic} if $d_M^K$ is a complete distance.
\end{definition}

As it is pertinent to our main result, we mention a result concerning the Kobayashi distance on bounded domains. Recall that a \textit{proper metric space} is one in which bounded sets are relatively compact.

\begin{proposition}
\label{ZimProp3.1}
Let $\Omega\subset\C^n$ be a bounded domain such that $\Aut(\Omega)$ acts cocompactly on $\Omega$. Then $(\Omega,d^K_{\Omega})$ is a proper metric space.
\end{proposition}

\begin{definition}
\label{TautDef} 
    A complex manifold $M$ is said to be \textit{taut} if for every complex manifold $N$, the set of all holomorphic functions from $N$ to $M$ is a normal family.
\end{definition}

\begin{remark}
\label{TautisHyperbolicRmk}
    It follows from the definitions that taut manifolds are always hyperbolic.
\end{remark}

\subsection{Rescaling}

To prove Theorem \ref{dissthm}, we will need to rescale a domain and look at its limit with respect to the local Hausdorff topology. Zimmer \cite{Zimmer} uses this method of rescaling for the proof of Theorem \ref{ZimThm1.1}, and we use it in much the same way. We refer the reader to \cite{Frankel89} for a detailed treatment of this technique.

\begin{definition}
\label{localhausdorffdef}
Here we define the \textit{local Hausdorff topology} on the set of all convex domains in $\C^n$. First, define the \textit{Hausdorff distance} between two compact sets $X,Y\subset\C^n$ by 
$$d_H(X,Y)=\max\left\{\max_{x\in X}\min_{y\in Y}\|x-y\|,\max_{y\in Y}\min_{x\in X}\|y-x\|\right\}.$$
To obtain a topology on the set of all convex domains in $\C^n$, we consider the \textit{local Hausdorff pseudodistances} defined by 
$$d_H^{(R)}(X,Y)=d_H\left(X\cap\overline{B_R(0)},Y\cap\overline{B_R(0)}\right), \ R>0.$$
Then a sequence of convex domains $\Omega_j$ converges to a convex domain $\Omega$ if there exists some $R_0\ge0$ such that 
$$\lim_{j\to\infty}d_H^{(R)}\left(\overline{\Omega_j},\overline{\Omega}\right)=0$$
for all $R\ge R_0$.
\end{definition}

\begin{remark}
The Kobayashi distance is continuous with respect to the local Hausdorff topology. 
\end{remark}

\begin{theorem}
\label{ZimThm4.2}
Suppose $\Omega_j\subset\C^n$ is a sequence of convex domains and $\ds\Omega=\lim_{j\to\infty}\Omega_j$ in the local Hausdorff topology. Assume the Kobayashi distance is nondegenerate on $\Omega$ and each $\Omega_j$. Then
$$d_{\Omega}^K(p,q)=\lim_{j\to\infty}d_{\Omega_j}^K(p,q)$$
for all $p,q\in\Omega$. Moreover, the convergence is uniform on compact subsets of $\Omega\times\Omega$.
\end{theorem}

We conclude this section with an example (see \cite{Zimmer}) that will appear in the proof of Theorem \ref{dissthm}.

\begin{example}(Example 4.4 in \cite{Zimmer})
\label{Zex4.4}
For $\alpha>0$, define
$$\mathcal{P}_{\alpha}=\left\{(z_1,\dots,z_n)\in\C^n\mid\re(z_1)>\alpha\sum_{j=2}^{n}|z_j|^2\right\}$$

Note that $\mathcal{P}_{\alpha}$ is biholomorphic to a ball.

Fix $r>0$, a sequence $r_j>0$ converging to $0$, and the sequence of linear maps
$$\Lambda_j(z_1,\dots,z_n)=\left(\frac{1}{r_j}z_1,\frac{1}{\sqrt{r_j}}z_2,\dots,\frac{1}{\sqrt{r_j}}z_n\right).$$
Then $$\mathcal{P}_{1/(2r)}=\lim_{j\to\infty}\Lambda_j(B_r(re_1))$$ in the local Hausdorff topology.
\end{example}

\subsection{Bergman Kernel}

The computations necessary for the proof of Theorem \ref{dissthm} require the following well-known description of the Bergman kernel. See Section 1.4 in \cite{Krantz} for more details and examples. 

\begin{definition}
\label{bergmankerneldef}
Let $\mu$ denote the standard Lebesgue measure on $\C^n$. For a domain $\Omega\subset\C^n$, let 
$$H^2(\Omega)=\left\{f\in\Hol(\Omega,\C) \ \bigg| \ \int_{\Omega}|f|^2d\mu<\infty\right\}.$$
Then $H^2(\Omega)$ is a Hilbert space. If $\{\phi_j\}$ is an orthonormal basis of $H^2(\Omega)$, then the function 
\begin{align*}
    &\kappa_{\Omega}:\Omega\times\Omega\to\C\\
    &\kappa_{\Omega}(z,w)=\sum_j\phi_j(z)\overline{\phi_j(w)}
\end{align*}
is called the \textit{Bergman kernel} of $\Omega$.
\end{definition}

The convergence of the series is absolute and uniform on compact subsets of $\Omega\times\Omega$, and for any $z\in\Omega$, the diagonal $\kappa_{\Omega}(z,z)$ is strictly positive. Additionally, $\kappa_{\Omega}$ does not depend on the choice of orthonormal basis.

The next result - Proposition \ref{ZimProp2.3} - is typically referred to as the monotonicity property. 

\begin{proposition}
\label{ZimProp2.3}
If $\Omega_1\subset\Omega_2\subset\C^n$ are domains, then 
$$\kappa_{\Omega_2}(z,z)\le\kappa_{\Omega_1}(z,z)$$
for all $z\in\Omega_1$.
\end{proposition}

Finally, we recall a property of the Bergman kernel \cite{Zimmer} that we will use later in the proof of Theorem \ref{dissthm}.

\begin{proposition}[Observation 2.5 in \cite{Zimmer}]
\label{ZimObs2.5}
Recall that for $\alpha>0$, we define 
$$\mathcal{P}_{\alpha}=\left\{(z_1,\dots,z_n)\in\C^n\mid\re(z_1)>\alpha\sum_{j=2}^{n}|z_j|^2\right\}.$$
There is a constant $C_{\alpha}>0$ such that
$$\kappa_{\mathcal{P}_{\alpha}}\left((\lambda,0,\dots,0),(\lambda,0,\dots,0)\right)=C_{\alpha}\re(\lambda)^{-(n+1)}$$
for all $(\lambda,0,\dots,0)\in\mathcal{P}_{\alpha}$.
\end{proposition}

\subsection{Bounded Symmetric Domains}

The proof of Theorem \ref{dissthm} - in particular, Proposition \ref{MRProof4} - requires results on the theory of bounded symmetric domains. We recall the necessary definitions and statements.

\begin{definition}
\label{symmetricdomaindef}
A bounded domain $\Omega\subset\C^n$ is called \textit{symmetric} if $\Aut(\Omega)$ is a semisimple Lie group which acts transitively on $\Omega$.
\end{definition}

\begin{definition}
Let $\Omega$ be a bounded symmetric domain. The \textit{real rank of $\Omega$} is the largest integer $r$ with the property that there exists a holomorphic isometric embedding $f:(\D^r,d^K_{\D^r})\to(\Omega,d^K_{\Omega})$.
\end{definition}

E. Cartan \cite{Cartan} characterized the bounded symmetric domains, showing that each one has real rank at least $1$. Moreover, the only bounded symmetric domain of rank $r=1$ is the ball.

Let $\Omega\subset\C^n$ be a bounded symmetric domain. Harish-Chandra constructed an embedding $F:\Omega\hookrightarrow\C^n$ whose image is convex and bounded (see \cite{Satake}, Chapter II, Section 4). Moreover, there exists a norm $\|\cdot\|_{HC}$ on $\C^n$ such that
\[
F(\Omega)=\left\{z\in\C^n \ \mid \ \|z\|_{HC}<1\right\}.
\]
We denote the image of this embedding by $F(\Omega)=\Omega_{HC}$. We say that $\Omega$ is in \textit{standard form} if it coincides with the image of its Harish-Chandra embedding, i.e., $\Omega=\Omega_{HC}$.

\begin{theorem}[Frankel, Nadel, \cite{Frankel}, \cite{Nadel}]
\label{FNSymmetric}
Let $X$ be a compact complex manifold with $c_1(X)<0$, and let $\widetilde{X}$ be the universal cover. If $\Aut\left(\widetilde{X}\right)$ is nondiscrete, then $\widetilde{X}$ is biholomorphic to either
\begin{enumerate}
    \item[(a)] a bounded symmetric domain, or
    \item[(b)] a nontrivial product $D_1\times D_2$, where $D_1$ is a bounded symmetric domain and $\Aut(D_2)$ is discrete.
\end{enumerate}
\end{theorem}

In the statement of Theorem \ref{FNSymmetric}, $c_1(X)$ refers to the first Chern class of $X$. See \cite{Chern}, \cite{Milnor}, and \cite{Zheng} for a full definition.

\section{A Localization Technique}
\label{LWSec}

In this section, we describe a localization technique of Lin and Wong \cite{LW} which we will use in the first part of the proof of Theorem \ref{dissthm}. In their paper, Lin and Wong use the technique to obtain a partial answer to a main problem in K\"ahler geometry: Does the universal cover of a compact K\"ahler manifold with negative sectional curvature admit a nontrivial, bounded holomorphic function? We refer the reader to \cite{LW} for the details of the following results, as well as other applications of the technique. 

We begin with some vocabulary that appears throughout the proof of Theorem \ref{dissthm}.

\begin{definition}
\label{boundaryneighborhood}
Let $\Omega$ be a domain on a complex manifold $M$, and let $p\in\partial\Omega$ be a fixed boundary point. A \textit{boundary neighborhood of $p$} is an open set $\widehat{\Omega}=\Omega\cap U$, where $U$ is a relatively compact coordinate chart centered at $p$.
\end{definition}

\begin{remark}
\label{BndryDomainisDomain}
Since $U$ is a relatively compact coordinate chart, its image under the chart map is a bounded open subset of $\C^n$. Therefore, a boundary neighborhood is biholomorphic to a bounded domain in $\C^n$.
\end{remark}

\begin{definition}
We say a point $p\in\partial\Omega$ is \textit{totally real} if there exists no complex analytic variety containing $p$ of positive dimension lying on $\partial\Omega$.
\end{definition}

\begin{definition}
\label{covercompactmanifolddef}
Recall that $\Aut(\Omega)$ is the group of biholomorphisms of $\Omega$. A domain $\Omega$ on a complex manifold $M$ is said to \textit{cover a compact manifold} if there exists a discrete group $\Gamma\le\Aut(\Omega)$ such that $\Gamma$ acts freely, properly discontinuously, and cocompactly on $\Omega$. 
\end{definition}

\begin{definition}
\label{compactquotientdef}
We say a domain $\Omega$ on a complex manifold $M$ \textit{admits a compact quotient} if $\Omega/\Aut(\Omega)$ is compact.
\end{definition}

\begin{remark}
A domain $\Omega$ admits a compact quotient if $\Omega$ covers a compact manifold.
\end{remark}

The following results describe the localization technique. The proofs (see \cite{LW}) require classical results of Montel and Cartan, as well as standard normal family arguments.

\begin{lemma}[Lemma 2.1 in \cite{LW}]
\label{LWLem2.1}
Let $\Omega$ be a domain on a taut manifold admitting a compact quotient. Then there exists a compact set $K\subset\Omega$ such that for every $y\in\Omega$ there is $t\in K$ and $g\in\Aut(\Omega)$ such that $g(t)=y$ (i.e., $\Aut(\Omega)\cdot K=\Omega$).
\end{lemma}

\begin{definition}
\label{fundamentalsetdef}
The compact set $K$ in Lemma \ref{LWLem2.1} is called the \textit{fundamental set} for $\Aut(\Omega)$.
\end{definition}

\begin{lemma}[Lemma 2.4 in \cite{LW}]
\label{LWLem2.4}
    Let $\Omega$ be a domain on a taut manifold $M$ admitting a compact quotient. Let $\{x_j\}\subset\Omega$ be a sequence converging to a boundary point $p\in\partial\Omega$. Then there exists a sequence $\{m_j\}\subset\Aut(\Omega)$ such that $m_j^{-1}(x_j)$, through a subsequence if necessary, converges to a point $x\in\Omega$. Moreover, $m_j(x)\to p$.
\end{lemma}

\begin{lemma}[Lemma 2.5 in \cite{LW}]
\label{LWLem2.5}
Let $\Omega$ be a domain on a taut manifold $M$, and suppose there is a totally real boundary point $p\in\partial\Omega$. Let $\{m_j\}\subset\Aut(\Omega)$ be a sequence such that $m_j(x)\to p$ for some $x\in\Omega$. Then for any compact set $K\subset\Omega$ and any boundary neighborhood $\widehat{\Omega}$ of $p$, $m_j(K)\subset\widehat{\Omega}$ for sufficiently large $j$. In particular, $m_j(y)\to p$ for any $y\in\Omega$.
\end{lemma}

The importance of the localization technique is in the power of Lemma \ref{LWLem2.5}. Using this technique, we can biholomorphically fit our domain $\Omega$ inside a boundary neighborhood $\widehat{\Omega}$ which, itself, is biholomorphic to a bounded domain in $\C^n$.

\section{Proof of Theorem \ref{dissthm}}
\label{ProofSec}

Here we present the proof of our main result. We first use the aforementioned rescaling argument of Frankel (Section \ref{GeoSec}) to show that $\Omega\subset M^n$ is biholomorphic to a domain $D\subset\C^n$ such that     
$$\mathcal{P}_{\alpha}\subset D\subset\mathcal{P}_{\beta}$$
for some $0<\beta<\alpha$, similar to Zimmer's use of the argument. However, in the construction of the rescaling maps, we use the localization technique from Section \ref{LWSec}. Next, we show that $\Aut(D)$ contains the one-parameter subgroup
$$u_t(z_1,\dots,z_n)=(z_1+it,z_2,\dots,z_n).$$
This step is a verification of Zimmer's calculations using the new rescaling maps. Most of the argument is technical, and, while necessary, is ultimately uninteresting.

Once we verify the computations, we appeal directly to Zimmer's results to complete the proof.

\subsection{Rescaling the Domain}

\begin{lemma}[Lemma 4.5 in \cite{Zimmer}]
\label{ZimLem4.5}
Let $\Omega\subset\C^n$ be a bounded domain with $C^{1,1}$ boundary.  After applying an affine transformation, we can assume that $0\in\partial\Omega$ and
$$B_r(re_1)\subset\Omega\subset B_1(e_1)$$
for some $r\in(0,1)$.
\end{lemma}
\begin{proof}
By translating, we can assume that $e_1\in\Omega$. Choose a boundary point $p_0\in\partial\Omega$ such that 
$$|p_0-e_1|=\max\{|p-e_1| \mid p\in\partial\Omega\}.$$
By rotating and scaling $\Omega$ about $e_1$, we may assume that $p_0=0$. Then $\Omega\subset B_1(e_1)$.

Next, for $p\in\partial\Omega$, let $n_{\Omega}(p)$ be the inward pointing normal unit vector at $p$. Since the boundary of $\Omega$ is $C^{1,1}$, there exists $r>0$ such that 
$$B_r(p+rn_{\Omega}(p))\subset\Omega$$
for every $p\in\partial\Omega$. So, since $0\in\partial\Omega$, by assumption, we have that $n_{\Omega}(0)=e_1$. Thus $B_r(re_1)\subset\Omega$.
\end{proof}

The following result is due to Deng et al., \cite{DGZ}; it is a higher dimensional analogue of Hurwitz's theorem.

\begin{theorem}[Theorem 2.2 in \cite{DGZ}]
\label{ZimThm2.6}
Suppose that $\Omega\subset\C^n$ is a bounded domain and let $x\in\Omega$. Let $f_j:\Omega\to\C^n$ be a sequence of injective holomorphic maps such that $f_j(x)=0$ for all $j$, and $f_j$ converges local uniformly to a map $f:\Omega\to\C^n$. If there exists $\varepsilon>0$ such that $B_{\varepsilon}(0)\subset f_j(\Omega)$ for all $j$, then $f$ is injective.
\end{theorem}

Recall that the set $\mathcal{P}_{\alpha}$ for $\alpha>0$ is defined to be
$$\mathcal{P}_{\alpha}=\left\{(z_1,\dots,z_n)\in\C^n\mid\re(z_1)>\alpha\sum_{j=2}^{n}|z_j|^2\right\}.$$
The main result of this section is the following:

\begin{proposition}
\label{MRProof1}
Let $M^n$ be a taut manifold, and let $\Omega\subset M$ be a subdomain with nonempty boundary, and assume $\partial\Omega$ is $C^{1,1}$. Suppose there exists a totally real boundary point $p\in\partial\Omega$. Suppose further that $\Omega$ covers a compact complex manifold. Then $\Omega$ is biholomorphic to a domain $D\subset\C^n$ such that
$$\mathcal{P}_{\alpha}\subset D\subset\mathcal{P}_{\beta}.$$
\end{proposition}

\begin{proof}
Suppose that $M^n$ is a taut manifold. Let $\Omega\subset M$ be a bounded domain with $C^{1,1}$ boundary admitting a compact quotient. Let $K\subset\Omega$ be the fundamental set of $\Aut(\Omega)$, i.e., $K$ is compact and $\Aut(\Omega)\cdot K=\Omega$. Let $p\in\partial\Omega$ be totally real, and let $\widehat{\Omega}$ be a boundary neighborhood of $p$. Note that $\widehat{\Omega}$ is biholomorphic to a bounded domain $\widetilde{\Omega}\subset\C^n$ (see Remark \ref{BndryDomainisDomain}). Let $\phi$ be the biholomorphism.

By Lemma \ref{ZimLem4.5}, we can assume that 
$$B_r(re_1)\subset\widetilde{\Omega}\subset B_1(e_1)$$
for some $r\in(0,1)$. Fix a sequence $p_j\in\widehat{\Omega}$ converging to $p$. By Lemma \ref{LWLem2.4}, there exists a sequence $\{m_j\}\subset\Aut(\Omega)$ such that $m_j(x)\to p$ for some $x\in\Omega$. Then by Lemma \ref{LWLem2.5}, for any compact set $L\subset\Omega$, $m_j(L)\subset\widehat{\Omega}$ for sufficiently large $j$. In particular, this holds for the fundamental set $K$. So, fix a sequence $k_j\in K$ such that $m_j(k_j)=p_j$.

Next, fix a sequence $r_j\in(0,r)$ converging to $0$. Since $\Aut(B_1(e_1))$ acts transitively on $B_1(e_1)$, we may pick $g_j\in\Aut(B_1(e_1))$ such that $g_j(\phi(p_j))=r_je_1$. The restriction of $g_j$ to $\widetilde{\Omega}$ is a biholomorphism from $\widetilde{\Omega}$ onto the domain $g_j(\widetilde{\Omega})\subset B_1(e_1)$. Consider the $\widetilde{\Omega}$-dilations
$$\Lambda_j(z_1,\dots,z_n)=\left(\frac{1}{r_j}z_1,\frac{1}{\sqrt{r_j}}z_2,\dots,\frac{1}{\sqrt{r_j}}z_n\right).$$
Let $\widetilde{\Omega}_j:=\Lambda_j(\widetilde{\Omega})$ and $F_j:=\Lambda_j\circ g_j\circ\phi\circ m_j:\Omega\to\widetilde{\Omega}_j$. Then 
$$\Lambda_j\left(B_r(re_1)\right)\subset\widetilde{\Omega}_j\subset\Lambda_j\left(B_1(e_1)\right).$$
Moreover, by Example \ref{Zex4.4}, 
$$\mathcal{P}_{\alpha}=\lim_{j\to\infty}\Lambda_j\left(B_r(re_1)\right), \ \text{where} \ \alpha=\frac{1}{2r}$$
and 
$$\mathcal{P}_{\beta}=\lim_{j\to\infty}\Lambda_j\left(B_1(e_1)\right), \ \text{where} \ \beta=\frac{1}{2}$$
in the local Hausdorff topology.

We claim now that, after passing to a subsequence if necessary, $F_j$ converges to a holomorphic embedding $F:\Omega\to\C^n$. Furthermore, if $D=F(\Omega)$, then 
$$\mathcal{P}_{\alpha}\subset D\subset\mathcal{P}_{\beta}.$$

By construction, $F_j(k_j)=e_1$, and the decreasing property of the Kobayashi pseudodistance implies that 
$$d^K_{\widetilde{\Omega}_j}(z,w)\le d^K_{\Lambda_j\left(B_1(e_1)\right)}(z,w)$$
for all $z,w\in\widetilde{\Omega}_j$. Theorem \ref{ZimThm4.2} implies that $d^K_{\Lambda_j\left(B_1(e_1)\right)}$ converges to $d^K_{\mathcal{P}_{\beta}}$ locally uniformly. So by the Arzel\'a-Ascoli theorem, we can pass to a subsequence where $F_j$ converges locally uniformly to a holomorphic map $F:\Omega\to\C^n$. 

Let $D=F(\Omega)$. Since 
$$\Lambda_j\left(B_r(re_1)\right)\subset\widetilde{\Omega}_j\subset\Lambda_j\left(B_1(e_1)\right)$$
for each $j$, it follows that 
$$\overline{\mathcal{P}_{\alpha}}\subset D\subset\overline{\mathcal{P}_{\beta}}.$$

It remains to show that $F$ is injective. We will do this using Theorem \ref{ZimThm2.6}. Since 
$$\mathcal{P}_{\alpha}=\lim_{j\to\infty}\Lambda_j\left(B_r(re_1)\right)$$
in the local Hausdorff topology and $\Lambda_j\left(B_r(re_1)\right)\subset\widetilde{\Omega}_j$ for each $j$, there exists $\varepsilon>0$ such that 
$$B_{\varepsilon}(e_1)\subset F_j(\Omega)$$
for each $j$. Passing to a subsequence if necessary, we can suppose $k_j\to k\in K$. Then consider the maps
$$G_j(q)=F_j(q)-F_j(k).$$
We will show that these maps are injective in order to invoke Theorem \ref{ZimThm2.6}.

Since 
$$\lim_{j\to\infty}F_j(k)=\lim_{j\to\infty}F_j(k_j)=F(k)=e_1,$$
$G_j$ converges locally uniformly to $F-e_1$. Moreover, by passing to another subsequence, we can assume that $\|e_1-F_j(k)\|<\frac{\varepsilon}{2}$ for each $j$. Then for every $j$, the map $G_j$ is injective, $G_j(k)=0$, and 
$$B_{(\varepsilon/2)}\left(0\right)\subset G_j(\Omega).$$
So by Theorem \ref{ZimThm2.6}, $F$ is injective. Thus $F$ is an embedding. 

Since $F$ is an embedding, $D$ is an open set, and so we have 
$$\mathcal{P}_{\alpha}=\Int(\overline{\mathcal{P}_{\alpha}})\subset D\subset \Int(\overline{\mathcal{P}_{\beta}})=\mathcal{P}_{\beta}.$$
\end{proof}

\subsection{The One-Parameter Subgroup}

We now show that $\Aut(D)$ contains the one-parameter subgroup.
\[
u_t(z_1,z_2,\dots,z_n)=(z_1+it,z_2,\dots,z_n)
\]

\begin{lemma}
\label{ZimLem4.7}
Let $\widetilde{\Omega}_j$ be as in the proof of Proposition \ref{MRProof1}. Suppose $z_j$ is a sequence such that $z_j\in\widetilde{\Omega}_j$ for each $j$, $z_j\to z$, and 
$$\liminf_{j\to\infty}d_{\widetilde{\Omega}_j}^K(e_1,z_j)<\infty.$$
Then $z\in D$.
\end{lemma}

\begin{proof}
Fix $z_0\in\Omega\subset M$ and let 
$$P=\max_{k\in K}d_{\Omega}^{K}(z_0,k).$$
Pick $j_s\to\infty$ such that 
$$Q=\lim_{s\to\infty}d_{\widetilde{\Omega}_{j_s}}^K(e_1,z_{j_s})<\infty.$$
Let $F_j:\Omega\to\widetilde{\Omega}_j$ and $k_j\in K$ be as in the proof of Proposition \ref{MRProof1}. Since $F_j(k_j)=e_1$, for each $s$, there exists
$$w_s\in\{y\in\Omega\mid d_{\Omega}^K(z_0,y)\le P+Q\}$$
such that $F_{j_s}(w_s)=z_{j_s}$. By Proposition \ref{ZimProp3.1}, $d_{\Omega}^K$ is proper. So we can pass to a subsequence where $w_s\to w\in\Omega$. Since $F_j\to F$ local uniformly, we have
$$F(w)=\lim_{s\to\infty}F_{j_s}(w_s)=\lim_{s\to\infty}z_s=z.$$
Thus $z\in F(\Omega)=D$. 
\end{proof}

The next lemma is highly technical, and serves only as a tangible representation of an obvious fact.  We refer the reader to \cite{Zimmer} for the proof.

\begin{lemma}[Lemma 4.8 in \cite{Zimmer}]
\label{ZimLem4.8}
Let $\widetilde{\Omega}$ be as in the proof of Proposition \ref{MRProof1}. For every $m>0$, there exists $\delta_m>0$ such that if $z_0\in\widetilde{\Omega}\cap B_{\delta_m}(0)$, $T>0$, and 
$$\{z_0+xe_1\mid-T<x<T\}\subset\widetilde{\Omega},$$
then
$$\left\{z_0+(x+iy)e_1 \ \Bigg| \ -\frac{T}{2}\le x\le\frac{T}{2}, -mT\le y\le mT\right\}\subset\widetilde{\Omega}.$$
\end{lemma}

\begin{proposition}
\label{MRProof2}
Let $D$ be as in Proposition \ref{MRProof1}. Then $\Aut(D)$ contains the one-parameter subgroup
$$u_t(z_1,\dots,z_n)=(z_1+it,z_2,\dots,z_n).$$
In particular, $\Aut(\Omega)\cong\Aut(D)$ is nondiscrete.
\end{proposition}
\begin{proof}
It suffices to fix $w_0\in D$ and $t\in \R$, and show that $w_0+ite_1\in D$. Since $F_j$ converges local uniformly to $F$, there exits $\varepsilon>0$ and $J\in\N$ such that 
$$B_{\varepsilon}(w_0)\subset F_j(\Omega)=\Lambda_j(\widetilde{\Omega})$$
for all $j\ge J$.

Define the sequence $\{w_j=\Lambda^{-1}_j(w_0)\}$. Then 
\begin{equation}
\label{MRP2eqn}
\{w_j+xe_1\mid -r_j\varepsilon<x<r_j\varepsilon\}\subset\Lambda_j^{-1}(B_{\varepsilon}(w_0))\subset\widetilde{\Omega} 
\end{equation}
whenever $j\ge J$. Here, the sequence $r_j$ is as in the proof of Proposition \ref{MRProof1}.

Fix $m\in\N$ such that $|t|<m\varepsilon$. Let $\delta_m$ be the associated constant from Lemma \ref{ZimLem4.8}. Since $r_j\to0$, we have that $w_j\to0$. So, by increasing $J$ if necessary, we may assume that 
$$w_j\in B_{\delta_m}(0)$$
whenever $j\ge J$. Then by (\ref{MRP2eqn}) and Lemma \ref{ZimLem4.8}, 
$$\left\{w_j+(x+iy)e_1 \ \Bigg| \ -\frac{r_j\varepsilon}{2}\le x\le\frac{r_j\varepsilon}{2}, -mr_j\varepsilon\le y\le mr_j\varepsilon\right\}\subset\widetilde{\Omega}$$
for all $j\ge J$. Thus
$$W:=\left\{w_0+(x+iy)e_1 \ \Bigg| \ -\frac{\varepsilon}{2}\le x\le\frac{\varepsilon}{2}, -m\varepsilon\le y\le m\varepsilon\right\}\subset\Lambda_j(\widetilde{\Omega}).$$
Then for all $j\ge J$,
$$d^K_{\Lambda_j(\widetilde{\Omega})}(w_0,w_0+ite_1)\le d^K_{W}(w_0,w_0+ite_1).$$
Thus $$\sup_{j\ge J}d^K_{\Lambda_j(\widetilde{\Omega})}(w_0,w_0+ite_1)<\infty.$$
By Lemma \ref{ZimLem4.7}, $w_0+ite_1\in D$.
\end{proof}

\subsection{Biholomorphic to the Ball}

Since we have verified all of the necessary calculations, we now appeal directly to Zimmer's results to complete the proof.

\begin{proposition}[Proposition 6.1 in \cite{Zimmer}]
\label{ZimProp6.1}
Suppose $\Omega\subset\C^n$ is a bounded domain which covers a compact manifold. If $\partial\Omega$ is $C^{1,1}$, then $\Omega$ is a bounded symmetric domain.
\end{proposition}

The proof of Proposition \ref{ZimProp6.1} exploits the geometry of the rescaled domain $D$, which we obtained in the proof of Proposition \ref{MRProof1}. By way of Theorem \ref{FNSymmetric}, a contradiction argument is used to show that the latter possibility of that result is impossible. The fact that $\Aut(D)$ contains a one-parameter subgroup allows for the construction of a seemingly infinite discrete subgroup of $\Aut(D)$. Zimmer deduces that this discrete group is compact, and hence finite giving the desired contradiction.

As the proof of Proposition \ref{ZimProp6.1} requires only that $\Omega$ be biholomorphic to the bounded domain $D\subset\C^n$, the next result follows immediately from Propositions \ref{MRProof1}, \ref{MRProof2}, and \ref{ZimProp6.1}.

\begin{proposition}
\label{MRProof3}
Let $\Omega$ be a subdomain of a taut manifold. Suppose $\partial\Omega$ is $C^{1,1}$, and that there exists a totally real boundary point $p\in\partial\Omega$. If $\Omega$ covers a compact manifold, then $\Omega$ is biholomorphic to a bounded symmetric domain in $\C^n$.
\end{proposition}

For the final part of the proof of Theorem \ref{dissthm}, we again cite Zimmer.

\begin{proposition}[Proposition 8.1 in \cite{Zimmer}]
\label{MRProof4}
Suppose $\Omega\subset\C^n$ is a bounded symmetric domain. If $\partial\Omega$ is $C^{1,1}$, then $\Omega$ is biholomorphic to the unit ball.
\end{proposition}

The proof of Proposition \ref{MRProof4} again exploits the geometry of the rescaled domain $D$, and makes use of the theory of bounded symmetric domains. Zimmer introduces a holomorphic function which measures the volume contraction (or expansion) of the biholomorphism from Proposition \ref{MRProof1} along a linear slice of the rescaled domain, $D$. A contradiction argument on the real rank of $\Omega$ as a bounded symmetric domain completes the proof. A key part of the argument is showing that one can parameterize the diagonal of a maximal polydisk in the Harish-Chandra embedding of $\Omega$ via the Jacobian matrix of the biholomorphism. Propositions \ref{ZimProp2.3} and \ref{ZimObs2.5} are instrumental here.

Theorem \ref{dissthm} now follows immediately from Propositions \ref{MRProof1}, \ref{MRProof2}, \ref{MRProof3}, and \ref{MRProof4}.

\section{Conclusions}
\label{ConclusionSec}

The initial inspiration for Theorem \ref{dissthm} was a result by Wong \cite{Wong81}:

\begin{theorem}[Theorem 1.5 in \cite{Wong81}]
\label{WongUni}
Let $M$ be a compact K\"ahler surface which is hyperbolic in the sense of Kobayashi. Suppose that 
\begin{enumerate}[label=\textup{(\arabic*)}]
    \item $M=\widetilde{M}/\Gamma$, where $\widetilde{M}$ is the universal cover of $M$ and $\Gamma$ is a discrete subgroup of the identity component of $\Aut\left(\widetilde{M}\right)$ acting freely on $\widetilde{M}$.
    \item $\Gamma$ is not isomorphic to the fundamental group of a compact real surface.
\end{enumerate}
Then $\widetilde{M}$ is biholomorphic to either the unit ball in $\C^2$ or the bidisk.
\end{theorem}

Proposition \ref{CFKWProp3.1} and, of course, Zimmer's paper \cite{Zimmer} were the inspiration for our method of proof.

A potential consequence of Theorem \ref{dissthm} is an extension of the following result of Cheung et al. \cite{CFKW}.

\begin{theorem}[Main Theorem in \cite{CFKW}]
\label{CFKWmain}
Let $M$ be a hyperbolic complex surface. Let $\Omega\subset\subset M$ be a subdomain with smooth boundary ($C^2$ is enough). If $\Omega$ admits a compact quotient, then either
\begin{enumerate}
    \item $\Omega$ is biholomorphic to a ball, or else
    \item The universal covering of $\Omega$ is biholomorphic to a bidisk.
\end{enumerate}
\end{theorem}

The proof of Theorem \ref{CFKWmain} is split into two parts. The first part is concerned with when $\partial\Omega$ contains a strictly pseudoconvex (totally real) point. In this case, we have Proposition \ref{CFKWProp3.1}, a stronger result which is true in any dimension. Our contribution was to replace the smooth boundary condition with the slightly weaker $C^{1,1}$ boundary condition.

The second part of the proof deals with the case when the boundary does not contain a strictly pseudoconvex point. In this case, the dimension $2$ condition is necessary, and the argument is slightly more involved, requiring work from \cite{FW1} and \cite{FW2}. This is the case that yields the bidisk. The argument uses estimates on invariant volume forms on the bidisk.

We replace the smooth boundary condition in Theorem \ref{CFKWmain} with the $C^{1,1}$ boundary condition to obtain the following result:

\begin{conjecture}
Let $M$ be a hyperbolic complex surface, and let $\Omega\subset\subset M$ be a subdomain with $C^{1,1}$ boundary. If $\Omega$ admits a compact quotient, and $\partial\Omega$ does not contain a totally real boundary point, then the universal cover of $\Omega$ is biholomorphic to a bidisk.
\end{conjecture}

The proof of this result will likely be similar to the one presented in Section 4 of \cite{CFKW} with some minor modifications to the argument where the defining function for $\partial\Omega$ is concerned. In particular, we choose the defining function $r$ to be a $C^{1,1}$ function rather than $C^{\infty}$. We suspect some technical issues arising from this alteration, but the argument should still follow much of the same logic.

In $n$ dimensions, the situation becomes much more complicated. We believe that in the absence of a totally real boundary point, we will obtain a universal cover biholomorphic to some higher rank symmetric space, or a compact manifold quotiented by a disk. Or we will be in the situation where, at each point, there will be a bidisk properly embedded in the universal cover. In any case, the argument will depend on the dimension of the disk sitting on the boundary of the domain $\Omega$.

\section*{Acknowledgement}
The author thanks Professor Bun Wong for his invaluable guidance during the early stages of this research. The author also extends his gratitude to the anonymous reviewer for their detailed suggestions and observations.

\end{document}